\documentclass[11pt]{amsart} 
\usepackage{amsmath,amssymb} 
\usepackage{graphicx,units} 
 
\def\p1i{\pi_1^{\infty}} 
\let\dfr=\rightarrow 
 
\let\ =\quad

\newtheorem{theorem} {Theorem} 
\newtheorem{definition}{Definition} 
 
\newtheorem{lemma}{Lemma} 
\newtheorem{proposition}{Proposition} 
\newtheorem{remark}{Remark}

\newcommand{\sm}{\smallsetminus}
\newcommand{\rel}{{\rm Rel}} 
\newcommand{\Z}{\mathbb Z} 
\title{On groups with linear sci growth}

\author{Louis Funar} 
\address{Institut Fourier BP 74, UMR 5582,  Universit\'e Grenoble I, 38402 
Saint-Martin-d'H\`eres Cedex, France } 
 
\email{louis.funar@ujf-grenoble.fr}

\author{Martha Giannoudovardi} 
\address{Department of Mathematics, 
University of Athens, 157 84 Athens, Greece} 
\email{marthag@math.uoa.gr} 
 
\author{Daniele Ettore Otera} 
 
\address{Institute of Mathematics and Informatics, Vilnius 
University, Akademijos str. 4, LT-08663, Vilnius, 
Lithuania} \email{daniele.otera@mii.vu.lt} 
\email{daniele.otera@gmail.com} 
 
\thanks{The authors acknowledge partial support from  
Research Funding Program Heracleitus II of the  
University of Athens and the European Union (M.G.), 
CMIRA Explora Pro 1200613701 (L.F.) and  
the European Commission's Marie Curie Intra-European Fellowship and INDAM of Italy (D.O.)}

\begin{document} 
 
%\maketitle 
 
\begin{abstract}
We prove that the semistability growth of hyperbolic groups is linear, which 
implies that hyperbolic groups which are sci (simply connected at infinity) 
have linear sci growth. Based on the linearity of the end-depth of finitely presented  
groups we show that the linear sci is preserved under amalgamated  
products over finitely generated one-ended groups.
Eventually one proves that most non-uniform lattices have linear sci.

\vspace{0.1cm} 
\noindent {\bf Keywords:} simple connectivity at
infinity, quasi-isometry,  end-depth, lattices in Lie groups,
amalgamated products. 
 
\vspace{0.1cm} 
\noindent {\bf MSC Subject:} 20 F 32, 57 M 50.
\end{abstract}

\maketitle

 \section{Introduction} 
 
 The metric spaces $(X,d_X)$ and $(Y,d_Y)$ are {\em quasi-isometric} 
 if there are constants $\lambda$, $C$ and maps $f: X \dfr Y$,  
$g: Y \dfr X$ (called $(\lambda ,C)$-quasi-isometries) 
 such that  the following: 
$$d_Y (f(x_1 ),f(x_2 ))\leqslant \lambda d_X (x_1 ,x_2)+C,\; d_X (g(y_1 ),g(y_2 ))\leqslant \lambda d_Y (y_1 ,y_2)+C,$$ 
$$d_X(gf(x),x)\leqslant C,\;  d_Y(fg(y),y)\leqslant C,$$ 
hold for all $x,x_1 ,x_2 \in X, y,y_1 ,y_2  \in Y$.

\begin{definition} 
 A  connected locally compact locally simply connected topological space $X$ with $\pi_1 X=0$ 
 is {\em simply connected at infinity} (abbreviated {\em sci} and one writes 
also $\pi_1^{\infty}X=0$)  
if for each compact $k\subseteq X$ there exists a larger compact 
$k\subseteq K\subseteq X$ such that any closed loop in $X \sm K$ is  
null homotopic in $X\sm k$. 
\end{definition} 
 
\noindent The sci is a fundamental tameness condition for non-compact spaces,  
as it singles out Euclidean spaces among contractible manifolds, following  
classical results of Stallings and Siebenmann. The concept of sci of finitely presented groups  goes back at least to  
Siebenmann's thesis (\cite{Sie}) and seems to first appear in its actual form in Houghton's paper \cite{H}  
(see also \cite{Mi1,Br1}), as follows: 
 
\begin{definition} 
A finitely presented  
group $G$ is simply connected at infinity  (abbreviated sci)  
if for some (equivalently any)  finite complex $X_G$ 
with $\pi_1 X_G=G$, its universal covering $\widetilde{X_G}$ is sci.  
\end{definition} 
 
The group $\Z^2$ is obviously not sci.  More interestingly, M. Davis (see e.g. \cite{DM2}) constructed  word hyperbolic  
groups $G$ (of virtual cohomological dimension $n\geq 4$  
by the results of Bestvina and Mess from \cite{bm})  
which are not sci.

All groups considered in the sequel will be finitely 
presented (unless the opposite is explicitly stated) and  a 
system of generators 
determines a word metric on the group. Although this depends on the 
chosen  generating set, 
the different word metrics are quasi-isometric. 
In \cite{FO} we enhanced the topological sci notion in the case of groups  
by taking advantage of this metric structure.

\begin{definition} 
Let $X$  be a sci non-compact metric space.  
The {\em sci growth} $V_X(r)$ (called rate of vanishing of $\p1i$ in \cite{FO})  
is the infimal $N(r)$ with the property that any loop in the complement   
of the metric ball $B(N(r))$ of radius $N(r)$ (centered at the identity) 
bounds a 2-disk outside $B(r)$. 
\end{definition} 
\begin{remark} 
It is easy to construct examples of metric spaces  
with arbitrarily large $V_X$. 
\end{remark} 
It is customary to introduce 
the following (rough) equivalence relation on real valued functions: 
the real functions $f$ and $g$ are equivalent, denoted by  $f\sim g$,  if 
there exist constants $c_i, C_j$ for $i=1,2,3$ (with $c_1,c_2 >0$) such that: 
\[ c_1 f(c_2 x)+ c_3 \leq g(x) \leq C_1 f(C_2 x) + C_3, \; {\rm for \; all\; } 
x. 
\] 
It is proved in \cite{FO} that 
the (rough)  equivalence class of $V_X(r)$ is a quasi-isometry invariant. 
In particular, if a finitely presented group $G$ is sci, then  
the (rough) equivalence class of the real function   
$V_G=V_{\widetilde{X_G}}$ is a quasi-isometry invariant of $G$, where 
${\widetilde{X_G}}$ is  the universal covering space 
of any finite complex $X_G$, with $\pi_1(X_G)=G$.

If a finitely presented group $G$ is sci and $V_G$ is a linear function we will say that $G$  
has {\em linear sci}. In contrast with the abundance of  
equivalence classes of geometric invariants of finitely presented groups  
(like group growth, Dehn functions or isodiametric functions), the metric  
refinements  
of topological properties seem highly constrained. We already found  
in \cite{FO}  
that many cocompact lattices in Lie groups and in particular  
geometric 3-manifold groups have linear sci.  
The aim of this paper is to further explore this phenomenon by considerably  
enlarging the class of groups with linear sci.  
Our first result is:   
 
\begin{theorem}\label{hyperbolicsci} 
Word hyperbolic groups which are sci have linear sci.
\end{theorem} 
 
 Let us recall that a group $G$   
is {\em one-ended} (or $0$-connected at 
infinity) if for any compact subset $L$ of the Cayley graph $X_G$ 
of $G$,  there exists a compact subset $K\supset L$ such that any 
two points out of $K$ can be joined by a path contained in $X_G 
\sm L$.

 If $H$ is a subgroup of two groups 
$G_1$  and $G_2$, the amalgamated product
$G_1 \ast_H G_2$ is the quotient of the free product of $G_1$ and $G_2$, where the copies of
$H$ in $G_1$ and $G_2$ are identified. If $H$ and $K$ are isomorphic subgroups of
$G_1$, the HNN-extension $G_1\ast_H$ is the quotient of
the free product $G_1\ast \langle t\rangle$,  where $H$ is identified with $t^{-1}Kt$ and
$\langle t \rangle$ denotes the free cyclic group generated by $t$. 
 
 We next show that the class of groups with linear sci 
is closed under amalgamated free products known to preserve the sci.

\begin{theorem}\label{amalgam} 
\begin{enumerate} 
\item Let $G_1$ and $G_2$  be one-ended finitely presented groups with linear sci and  $H$
be a finitely generated subgroup of $G_1$ and $G_2$ with one end. Then 
the amalgamated free product  $G=G_1\ast _H G_2$ has linear sci. 
\item  
Let $G_1$ be a finitely presented group with linear 
sci and  $H, K$ be isomorphic finitely generated subgroups with one end.   
Then the HNN-extension $G= G_1\ast_{H}$ has linear sci. 
\end{enumerate} 
\end{theorem} 
 
Theorem \ref{amalgam}  
is similar to (but under stronger restrictions than)  
the results obtained by Mihalik and Tschantz  
in \cite{MiTsch,MiTsch3} in the context of semistability.  
Notice that the sci is {\em not} preserved under amalgamated products over  
multi-ended subgroups (see \cite{Ja2}).

Previous results of \cite{FO} dealt with all cocompact lattices  
in connected Lie groups but the solvable ones.  
We now consider non-uniform  
lattices.  Our main result in this direction is:   
 
\begin{theorem}\label{hrank} 
\begin{enumerate} 
\item Let $G$ be a semisimple Lie group for which the associated  
symmetric space $G/K$ is of dimension $n\geq 4$ and of 
$\mathbb R$-rank greater than or equal to  $2$. Let $\Gamma$ be an 
irreducible, non-uniform lattice in  $G$ of $\mathbb Q$-rank one. 
Then $\Gamma$ is sci with linear sci. 
\item Every lattice $\Gamma \subset SO(n,1)$, $n\geq 2$  has linear sci. 
\end{enumerate} 
\end{theorem} 
 
\begin{remark} 
We believe that the last result also holds for $\mathbb R$-rank $1$ semisimple Lie 
groups and for non-uniform lattices of $\mathbb Q$-rank $>1$. 
Further, if moreover strongly polycyclic groups had 
linear sci then all lattices in Lie groups (of sufficiently large dimension) would have linear sci. 
\end{remark} 
 
The results of this paper naturally lead to the  
question of the existence of sci groups with super-linear sci growth.  
It seems still unknown whether  
sci CAT(0) groups have linear sci.

\vspace{0.2cm} 
 
\noindent {\bf Acknowledgements}: 
The authors are indebted to  Y. de Cornulier, F. Haglund, P.  
Papazoglou, V. Poenaru, M. Sapir and A. Valette  
for useful discussions and advice and the referee for corrections of both mathematical and historical nature.

\section{Proof of Theorem  \ref{hyperbolicsci}}
\subsection{Preliminaries on hyperbolic groups} 
 
Let $(X,d)$ be a geodesic metric space, which in our case will be  
the Cayley graph of a finitely generated group $G$ endowed with the word metric 
induced by a finite generating system. 
Let $\gamma$ be a geodesic path in $X$, possibly infinite. For  
any $x$, $y\in\gamma$, we  denote by  
$[x,y]_\gamma$ the sub-path of $\gamma$ that connects $x$ to $y$. When $\gamma$ 
is finite, we denote by $\ell(\gamma)$ the length of the path $\gamma$. 

A geodesic triangle in $X$ is $\delta$-{\em slim}  
if every side is contained in the $\delta$-neighbourhood of the  
union of its other sides. The group $G$ is $\delta$-\textit{hyperbolic}, for 
some $\delta\geqslant 0$,  
if all geodesic  
triangles in $X$ are $\delta$-slim. The group $G$ is \textit{hyperbolic} if it 
is  
$\delta$-hyperbolic for some  
$\delta\geqslant0$. It is well known that hyperbolic groups are finitely presented.

Suppose from now on that $G$ is a $\delta$-hyperbolic group and  
$X_G$ its Cayley complex associated  
with a finite presentation  $\mathcal{P}=\langle  
S\mid R\rangle$. We will only consider geodesics within the  
Cayley graph, namely the 1-skeleton $X_G^{(1)}$ of $X_G$. Notice that  
while the Cayley complex may change when adding words equal to  
the identity to the relators in $\mathcal{P}$, the Cayley graph remains  
unchanged.  
  
Bestvina and Mess \cite{bm} proved the following crucial fact, which was 
pointed out by Mihalik: 
\begin{proposition}[\cite{bm}]\label{grays} 
Let $G$ be a hyperbolic one-ended group. There is a constant 
$c\geqslant 0$ so that for all 
$x\in X_G$ there exists an infinite geodesic ray issuing from the identity of $G$ 
which passes within $c$ of $x$.   
\end{proposition}
We say that two geodesic rays are \textit{asymptotic} if  
their images in $X_G$ are at a finite Hausdorff distance apart.  
This defines an equivalence relation on the collection of geodesic rays in $X_G$. 
 The \textit{boundary} $\partial X_G$  of $X_G$ is the collection  
of equivalence classes,  under this relation, of geodesic rays in $X_G$.  
Unless  otherwise stated, all geodesics considered will be assumed to be unit 
speed geodesics.  

We say that $X_G$ satisfies $\rel(M)$ for some $M>0$ if there exists $L>0$ such 
that 
for all $R>0$ and $x,y\in X_G$ with $d(x,1)=d(y,1)=R$ and $d(x,y)\leq M$, there 
exists a path of length at most $L$ that connects $x$ and $y$ outside the 
ball $B(R-c)$, where $c$ is the constant of Proposition \ref{grays}. It will be 
convenient to say that then $X_G$ satisfies $\rel(M)$ with constant $L$. The 
significance of $\rel(M)$ is revealed in the following:
\begin{proposition}[\cite{bm}]\label{bmCP}
Let $G$ be a hyperbolic one-ended group. If $\rel(M)$ fails for some $M>0$, 
then 
$\partial G$ contains a global cut point.
\end{proposition}
Combined with a result of Bowditch, Svenson and Swarup 
\cite{bb,Sv,Sw}  which states that $\partial X_G$ has no global cut points, we 
derive that every one-ended hyperbolic group satisfies $\rel(M)$, for any $M>0$.
 
\subsection{Semistability and simple connectivity at infinity} 
The aim of this section is to put the sci growth into a more general context.  
Recall that a {\em ray} in a noncompact topological space $X$ is a proper map 
$\gamma:[0,\infty)\to X$. As above we consider $X$ to be connected locally compact 
locally simply connected topological space. 
Two rays $\gamma_1$ and $\gamma_2$  converge to the 
{\em same end} of $X$ if for any compact $C\subset X$ there exists  
$R$ such that $\gamma_1([R,\infty))$ and $\gamma_2([R,\infty))$ lie in the 
same component of $X\sm C$. The set of rays under this equivalence 
relation is the same as the set of ends of $X$.  
 
\begin{definition} 
An end of $X$ is {\em semistable} if any two rays of  
$X$ converging to this end 
are properly homotopic. This is equivalent (see e.g. \cite{G})  
to the following: for 
any ray $\gamma$ converging to the end and for any $n\geq0$ there exists 
$N\geq n$ such that any loop based on a point of $\gamma$ with image outside  
the metric ball $B(N)$ of radius $N$ and fixed center  
can be pushed (rel $\gamma$) to infinity by a homotopy in $X\sm B(n)$. 
\end{definition} 
 
A topological space is semistable if all its ends are semistable.  
This definition was extended to groups:
A finitely presented  
group $G$ is {\em semistable}   
if for some (equivalently any)  finite complex $X_G$ 
with $\pi_1 X_G=G$ its universal covering $\widetilde{X_G}$ is  
semistable.    
 
Many classes of groups are known to be semistable (see e.g.  
\cite{MiTsch,Mi3} and also \cite{bb,Sv,Sw} for the case of  
hyperbolic groups) but examples of  
finitely presented groups which are {\em not} semistable 
are still unknown.  
There is a well-defined notion of  
topological fundamental group at infinity  
associated to a semistable end of a group (see \cite{GM}).  
Now, following \cite{FO}, we consider the following metric  
refinement of the semistability: 
 
\begin{definition} 
Let $X$ be a non-compact metric space, $e$ an end of $X$ and $\gamma$ a 
ray converging to $e$. The {\em semistability growth} function $S_e(r)$ is 
the infimal $N(r)$ with the following property: for any 
$R\geq  N$ and any loop $l$ based on $\gamma$ which lies in $X\sm B(N)$ 
there exists a homotopy rel $\gamma$ supported in $X\sm B(r)$   
which moves $l$ to a loop in $X\sm B(R)$.  
\end{definition}

Set $S_G$ for $\sup_eS_e$, with $e$ running over the set of ends of $\widetilde{X_G}$, where 
$X_G$ is a finite complex with fundamental group $G$, whenever this 
is defined. It is not difficult to see that the equivalence class  
of $S_G$ is a well-defined quasi-isometry invariant of the  
finitely presented group $G$.

The principal result of this section is the following immediate 
 connection between  
sci growth and semistability growth:  
 
\begin{proposition}\label{ineq} 
Assume that $G$ is a finitely presented sci group.  
Then $V_{G} =  S_{G}$. 
\end{proposition} 
\begin{proof} 
For given $r$ as the space $\widetilde{X_G}$ is sci there  
exists some large enough $N(r)$ so that any loop 
within $\widetilde{X_G}\sm B(N(r))$ bounds a disk outside $B(r)$.  
Let $l$ be a loop not intersecting $B(S_G(r))$. By the  
semistability assumption one can homotope $l$ in $\widetilde{X_G}\sm B(r)$ 
to a loop $l'$ lying  within $\widetilde{X_G}\sm B(N(r))$. But $l'$ bounds a disk  
outside $B(r)$ and hence $l$ bounds a disk outside $B(r)$.  
This proves that $V_G(r)\leq S_G(r)$. 
 
For the reverse inequality let $l$ be a loop based at  
$\gamma(V_G(r)+\varepsilon)$ (for arbitrarily small $\varepsilon$)  
outside $B(V_G(r))$, where $\gamma$ is a given ray. Then $l$  
bounds a disk outside $B(r)$, which yields a nullhomotopy of the based  
loop $l$ to  the base point $p$. We push then $p$ along $\gamma$  
as far as we want. This proves that $S_G(r)\leq V_G(r)$.  
\end{proof} 
 
It follows that Theorem \ref{hyperbolicsci} is an immediate consequence of the more general: 

\begin{theorem}\label{hyperbolic}
The semistability growth of word hyperbolic groups is linear. 
\end{theorem}

\subsection{Proof of Theorem \ref{hyperbolic}} 
 
Consider first the case when  
$G$ is a one-ended hyperbolic group. Let $\delta$ be the hyperbolicity constant for a Cayley complex $X_G$ of $G$, and let $c$ be the constant 
provided by Proposition \ref{grays}. By Proposition \ref{bmCP} 
and the absence of cut points (from \cite{bb,Sv,Sw}), $X_G$ satisfies 
$\rel(M)$ for some $M>6c+2\delta+3$ with constant $L>2c+4$.  
Without loss of generality we can assume 
that $X_G$ is associated with a presentation of $G$ that 
contains as relators all words of length less than  
$2L+4c$ which are equal to the identity in $G$. In this section, unless stated otherwise, the balls we consider will be centered at the identity.
 
Let $n\in\mathbb{Z}_+$ and $\gamma$ be a geodesic ray in $X_G$ that 
starts from identity. We will show that every  
loop $f$ based at a point $x\in\gamma$ and which lies outside $B(n+2c)$, can be 
pushed rel $\gamma$ arbitrarily far away by a homotopy outside $B(n)$.

Let $p$, $q$ be two adjacent vertices of $f$ and $r=d(p,1) > n+2c$.  
There exist unit speed geodesic rays $\gamma_0$ and $\gamma_1$ issuing from the  
identity which pass within $c$ of $p$ and $q$, respectively.

We now establish following (\cite{bm}, Prop.3.2) the following lemma:
\begin{lemma}\label{cll1}
For every integer $i\geqslant 0$, there 
exists a path $f_i$ from 
$\gamma_0(r+i)$ to 
$\gamma_1(r+i)$ such that:
\begin{enumerate}
 \item The path $f_i$ lies outside $B(r+i-c)$;
 \item For any $j\in\{0,1,\ldots,L^i\}$, there is a unit speed geodesic ray, 
$\gamma_{\frac{j}{L^i}}$ issuing from the identity such that:
\[\gamma_{\frac{j}{L^i}}(r+i)\in f_i,\]
and for $j < L^i-1$,
\[d\left(\gamma_{\frac{j}{L^i}}(r+i),\gamma_{\frac{j+1}{L^i}}
(r+i)\right)\leq M.\] 
\end{enumerate}
\end{lemma}
\begin{proof} 
We use induction on $i$. 
Using the triangle inequality we obtain that   
$d(\gamma_0(r),p)\leq 2c$ and 
$d(\gamma_1(r),q)\leq 2c+1$, so that $d(\gamma_0(r),\gamma_1(r))\leq 
4c+2<M$. By property $\rel(M)$, there is a path $f_0$ 
of length at most $L$ which joins $\gamma_0(r)$ to $\gamma_1(r)$ outside $B(r-c)$, 
hence the claim holds for $i=0$.   

Assume now that the result holds for some $i\geq 0$. 
By property $\rel(M)$, for any $j\in\{0,1,\dots,L^i-1\}$,  there is a path 
$\alpha_j:[0,L]\to X_G$, of at most unit speed, that joins 
$\gamma_{\frac{j}{L^i}}(r+i)$ to $\gamma_{\frac{j+1}{L^i}}(r+i)$ and lies 
outside $B(r+i-c)$.

Next, for any $k\in\{1,\ldots,L-1\}$, there exists a 
geodesic ray $\gamma_k'$, issuing from 
$1$ that passes within $c$ of $\alpha_j(k)$. Let $y_k\in
\gamma_k'$ be a closest point to 
$\alpha_j(k)$. Then $d(y_k,\alpha_j(k))\leq c$, so that $d(y_k,1)\geq r+i-2c$. 
It follows that there exists $z_k\in \gamma_k'$ 
such that $d(z_k,1)\geq r+i+1$ and $d(z_k,y_k)\leq 2c+1$. We obtain therefore that  
$d(z_k,z_{k+1})\leq 6c+3$. As the 
geodesic triangle of vertices $1,z_k,z_{k+1}$ is $\delta$-slim we derive that: 
\[d(\gamma_k'(r+i+1),\gamma_{k+1}'
(r+i+1))\leq 6c+2\delta+3<M.\]
Therefore, from $\rel(M)$ there exists a path of length at most 
$L$ which lies  outside $B(r+i+1-c)$ and joins 
$\gamma_k'(r+i+1)$ to 
$\gamma_{k+1}'(r+i+1)$. 
The concatenation of these paths yields the desired $f_{i+1}$. We set $\gamma_{\frac{jL+k}{L^{i+1}}}$ to be the geodesic ray $\gamma_k'$, thereby concluding the induction step. Notice that:
\[\ell([\gamma_{\frac{j}{L^i}}(r+i),\gamma_{\frac{j+1}{L^i}}
(r+i)]_{f_i})\leq 
L.\]
\end{proof}

Let $P, Q$ be two geodesic arcs joining $p$ and $q$ to 
$\gamma_0(r)$ and $\gamma_1(r)$, respectively. For any $N\geq 0$, let 
$\Phi_N(p,q)$ be the closed oriented loop obtained by concatenation of  
$P$, $\gamma_0([r,r+N])$, 
$f_{N}$, $\gamma_1([r,r+N])^{-1}$, $Q^{-1}$ and the edge $qp$.  
 
\begin{lemma}\label{cl2} 
The loop $\Phi_N(p,q)$ is null homotopic outside $B(n)$, for $N\geq 0$. 
\end{lemma} 
\begin{proof}
The closed  loop $\Phi_0(p,q)$ has length at most $L+4c+2$ and 
lies outside $B(r-c-1)$. By our hypothesis on the group presentation, 
this loop bounds a 2-cell $A_0$ in the Cayley complex, thereby proving 
the claim for $N=0$.

Now, for any $i\in \{0,1,\ldots,N-1\}$, let $\Phi_i$ be the concatenation  of  the four paths $\gamma_0([r+i,r+i+1])$, $f_{i+1}$, $\gamma_1([r+i,r+i+1])^{-1}$, and $f_{i}^{-1}$. Then, $\Phi_N(p,q)$ is homotopic to the product of $\Phi_0(p,q)$   
and $\Phi_0\Phi_{1}\cdots \Phi_{N-1}$. 
We can further  decompose each $\Phi_i$ as the composition of loops 
$\Phi_i(j)$, consisting of the concatenation of the 
following four paths: $\gamma_{\frac{j}{L^i}}([r+i,r+i+1])$, 
$[\gamma_{\frac{j}{L^i}}(r+i+1), \gamma_ 
{\frac{j+1}{L^i}}(r+i+1)]_{f_{i+1}}$, 
$\gamma_{\frac{j+1}{L^i}}([r+i,r+i+1])^{-1}$, and 
$[\gamma_{\frac{j}{L^i}}(r+i), \gamma_ 
{\frac{j+1}{L^i}}(r+i)]_{f_{i}}^{-1}$.

Recall from above that $\ell([\gamma_{\frac{j}{L^i}}(r+i),\gamma_{\frac{j+1}{L^i}}
(r+i)]_{f_i})\leq 
L.$ For fixed $i$ and $j$, let $\alpha_j:[0,L]\to X_G$ denote a path with image:
$$a_j([0,L]) = [\gamma_{\frac{j}{L^i}}(r+i),\gamma_{\frac{j+1}{L^i}}
(r+i)]_{f_i}.$$
So then $a_j$ has at most unit speed and lies outside $B(r+i-c)$. For any 
$k\in\{0,\ldots,L-1\}$, denote by $y_k$ 
a point on the geodesic $\gamma_k'=\gamma_{\frac{jL+k}{L^{i+1}}}$ which is closest 
to $a_j(k)$, and by $\beta_k$ a geodesic path that joins 
$\alpha_j(k)$ to $y_k$. Then, $d(y_k,a_j(k))\leq c$ and 
$$r+i-2c\leq d(y_k,1)\leq \frac{L}{2}+r+i+c.$$
It follows that $d(y_k,\gamma_k'(r+i+1))\leq \frac{L}{2}+c+1$ and the path $[y_k,\gamma_k'(r+i+1)]_{\gamma_k'}$ lies outside $B(r+i-2c)$. Therefore, the closed loop obtained by concatenating 
$\beta_k$, $[y_k,\gamma_k'(r+i+1)]_{\gamma_k'}$, $[\gamma_k'(r+i+1),\gamma_{k+1}'(r+i+1)]_{f_{i+1}}$, $[\gamma_{k+1}'(r+i+1),y_{k+1}]_{\gamma_{k+1}'}$, $\beta_{k+1}$,  and $[\alpha_j(k),\alpha_j(k+1)]_{\alpha_j}$ 
has length at most $2L+4c-1$ and lies outside $B(r+i-2c)$. Hence, it 
bounds a 2-cell $A_{i,j}(k)$ in the Cayley complex outside $B(n)$.
The union $A_{i,j}$ of all these 2-cells  
$\cup_{k\in\{0,\ldots,L-1\}}A_{i,j}(k)$ is the image of a disk filling  
the loop $\Phi_i(j)$ outside $B(n)$. This proves the Lemma. 
\end{proof}

The composition of all loops of the $\Phi_N(p,q)$, for $p,q$ successive 
vertices of the loop $f$ is then freely homotopic within $X_G\sm B(n)$ 
to a loop consisting of concatenations of paths of the form 
$f_{N}$, which lay outside  $B(r+N-c)$. It follows that for any $N\geq0$  
our initial loop $f\subset X_G\sm B(n+2c)$ is homotopic 
rel $\gamma$ to a loop in $X_G\sm B(r+N-c)$ by a homotopy outside $B(n)$.

Eventually, when $G$ is not one-ended, we work on the 
connected components of  $X_G\sm B(n+2c)$. 
This proves Theorem \ref{hyperbolic}.

\section{Proof of Theorem \ref{amalgam}}
\subsection{Preliminaries on the end-depth}
The sci and its refinement (the sci growth rough equivalence class)  
are 1-dimensional invariants at infinity for a group 
$G$, in the sense that they take care of loops and disks. The 
0-dimensional analogue of the simple connectivity at 
infinity is the connectivity at infinity, namely the one-endedness. 
One could adapt the notion of sci growth to the  
growth of an end. This was already considered by Cleary and Riley  
(see \cite{CR}). This leads to the following metric refinement which is  
the $0$-dimensional counterpart of the sci growth:  
 
\begin{definition} 
Let $X$ be a one-ended metric space. The {\em end-depth}  $V_0(X)$ of $X$  is 
the infimal $N(r)$ with the property that any two points which sit 
outside the ball $B(N(r))$ of radius $N(r)$ can be joined by a 
path outside $B(r)$. 
 
If $G$ is a finitely generated one-ended group then the end-depth of $G$  
is the (rough) equivalence class of the real function   
$V_{0,G}=V_0(X_G)$,  
where ${X_G}$ is  a Cayley graph of $G$ associated to a  
finite generating set.  
\end{definition} 
One can define in the same way the end-depth of a specific end  
of a space or finitely generated group which are not necessarily one-ended.  
In \cite{Ot-08} one proved that the  
(rough) equivalence class of $V_{0,G}$ is a well-defined quasi-isometry 
invariant of one-ended finitely presented groups. 
Examples of groups whose Cayley graphs have {\em dead-ends} (i.e.  
end-depth functions strictly larger than $x+c$, for any $c$)  
were obtained in \cite{CR}.  
Our second result shows that the (rough) equivalence class  
of the end-depth is not meaningful:

\begin{proposition}\label{enddepth} 
Every finitely generated one-ended group has linear end-depth. More  
precisely we have the inequality:  
$$V_0(X_G)(r)\leq 2r, \; for \; large\; enough \; r,$$  
where $X_G$ denotes the Cayley graph associated to a finite generating set of the group $G$. 
\end{proposition} 

We postpone the proof of this Proposition to the end of this section. 

\subsection{End of the proof of Theorem \ref{amalgam}} 
Consider the amalgamated  
product $G = G_1 \ast _H G_2$. Let $X_1$ and $X_2$ be the  
standard $2$-complexes associated to some finite  
presentations of $G_1$ and $G_2$, respectively. 
Let $S_H$ be a finite set of generators of $H$ which are represented  
by a wedge of loops $Y$ in both $X_1$ and $X_2$.   
The space $X$ obtained by attaching 
$X_1$ and $X_2$ along $Y$ has fundamental group $G$. 
Let $C_H$ be the Cayley graph of $H$ corresponding to the generators $S_H$.  
The image of $\widetilde Y$ in $\widetilde X_i$ is then  
homeomorphic to $C_H$.   
 Furthermore the universal covering $\widetilde{X}$ is constructed from coset  
copies of the universal coverings $\widetilde X_1$ and $\widetilde X_2$  
which are attached along copies of $C_H$. 
 
We consider a metric ball $B(r)$ of radius $r$ in $\widetilde X$   
centered at a fixed point. By compactness, $B(r)$ intersects  
only finitely many copies of $\widetilde X_1$ and $\widetilde 
X_2$. Since $\widetilde X_1$ and $\widetilde X_2$  
have linear sci there exists a constant $c$ such that any  
loop lying in one copy of either $\widetilde X_1$ or $\widetilde X_2$   
which is outside $B(cr)$ is contractible by a nullhomotopy outside $B(r)$.

Since the one-ended group $H$ has linear end-depth by Proposition \ref{enddepth},  
one can find a constant $c_1$ such that any two points 
of a copy of $C_H$ lying outside $B(c_1r)$ can be connected by a path  
within that copy $C_H$ not intersecting $B(cr)$.

The proof that any loop of $\widetilde X$ which lies  outside  
$B(c_1r)$ bounds a disk outside $B(r)$ is now standard following 
\cite{Ja2}. Any edge loop $L$ starting at $g\in G$ 
can be written as a word $ga_1a_2a_3 \cdots a_n$, with  $a_i \in 
G_1$, when $i$ is odd and  $a_i \in G_2$, when $i$ is even,  
such that the equality  
$a_1a_2 \cdots a_n =1$ holds in $G$.  
The structure theorem for amalgamated products implies that there exists   
some $i$ so that $a_i\in H$ (see 
\cite{LS}). Thus the edge sub-path $l$ corresponding to the  
element $a_i\in H$ starts and ends in the same copy of $C_H$. 
 
We will show that $l$ can be homotoped in $\widetilde X$ rel end points  
into this copy of $C_H$.   As $L$  lies outside $B(c_1r)$,  
the end points of $l$ are outside $B(c_1r)$ and by the above argument  
they can be connected by some path $p$ lying within the same copy of  
$C_H$ and which does not intersect $B(cr)$. 
The resulting loop $l\cup p$ obtained by gluing together $l$ and $p$ at their  
common end points is  therefore contained in one  
copy of  either $\widetilde X_1$ or else of $\widetilde X_2$.   
Moreover, $l\cup p$  lies in the complement of $B(cr)$.  
By hypothesis, $G_i$ have linear sci and thus $l\cup p$  
can be contracted out of $B(r)$. This establishes the claim.  
The word associated to the path $p$ belongs to $H$ and  
it can be absorbed into $a_{i-1}$.  
Thus  we obtain a free homotopy of $L$ outside $B(r)$ to a  
loop $L'$ starting at $g$ which corresponds  
to a word strictly shorter than that of $L$.  
Then, by induction on $n$, we can decrease the length  
$n$ until the resulting loop has $n=1$.  
This proves the first part of Theorem \ref{amalgam}.

In order to prove the second part let us recall the HNN construction.  
If $H$ is a finitely generated subgroup of the finitely presented group  
$G_1$ and $f:H\to G_1$ is a monomorphism from $H$ into $G_1$ 
we set $K=f(H)$. Suppose that $H$ is generated by $a_1, \ldots , a_n$ 
and denote by $c_i$ the generators $f(a_1),f(a_2),\ldots f(a_n)$ of $K$. Let  
$$\langle b_1,\ldots , b_m, a_1, \ldots , a_n, c, \ldots , c_n \ | \ p_1=1, 
\ldots , p_k=1\rangle $$  
be a presentation for $G_1$. Then the  HNN-extension $G=G_1\ast_H$   
of $G_1$ by $f$ has the presentation:  
$$\langle b_1, \ldots , b_m, a_1, \ldots , a_n, c_1, \ldots , 
c_n, t \mid p_i=1\text{ for }i\leq k, c_j=t^{-1} a_j t\text{ for }j\leq n \rangle.$$

Consider the $2$-complex $X_1$ associated to the given presentation  
of $G_1$. It contains two wedges of circles $Y_H,Y_K$ associated to 
finite sets of generators of $H$ and $K$.  
Consider the space $X$ obtained from a copy 
of $X_1$ and a copy of $Y_H\times [0,1]$ where 
$Y_H \times \{0\}$ is identified with the copy of $Y_H$ in $X_1$ and  
$Y_H \times \{1\}$ is identified with the copy of $Y_K$ in  
$X_1$ by means of $f$. The universal covering space $\widetilde X$ of $X$ can be 
 constructed from coset copies of $\widetilde X_1$ and  
$C_H \times [0,1]$, where $C_H$ denotes the Cayley graph of $H$.  
As above $C_H$ is the image of $\widetilde Y_H$ inside $\widetilde X_1$.

As in the case of an amalgamated product above,  
the key tool is Britton's lemma giving the structure of an HNN extension which we state as follows.    
Suppose that we have the equality  
$g_0 t ^{i_1}g_1 t^{i_2}\cdots t^{i_n}g_n=1$ in $G$,  
where $g_k \in G_1$. Then, for some $k$, either $i_k 
>0$, $i_{k+1}<0$, and $g_k$ is in $K$ or else $i_k 
<0$, $i_{k+1}>0$, and $g_k$ is in $H$.  
 
We denote by $B(r)$ the metric ball centered at the identity in $\widetilde X$.  
Since each copy of $\widetilde X_1$ has linear sci  
there is $c$ such that any loop in $\widetilde X_1$ outside  
the metric ball $B(cr)$   
contained in one copy of $\widetilde X_1$ bounds a disk not 
intersecting $B(r)$. The metric ball $B(cr)$ intersects only  
finitely many copies of $\widetilde Y_H \times [0,1]$.  
By Proposition \ref{enddepth}, since $H$ is one-ended, one can choose $c$ large enough such that  any two points of one copy of $\widetilde Y_H \times 
[0,1]$ which lie outside $B(cr)$ can be joined by a path  within this copy,  
not intersecting $B(r)$. 
 
Let $L$ be an edge loop in  $\widetilde X \sm  B(c^2r)$. This loop can   
be represented by a word $g_0 t ^{i_1}g_1 t^{i_2} \cdots t^{i_n}g_n$,  
where $g_j \in G_1$ and which is equal to $1$ in $G$.  
If $\sum_{j=1}^n |i_j|=0$ then the loop  
is contained in one copy of $\widetilde X_1$ and thus is  
contractible out of $B(r)$, by hypothesis.  
When $\sum_{j=1}^n |i_j|>0$, let $k$ be the one provided by   
Britton's lemma in the form stated above. Then the edge path corresponding  
to the word $t^{{\rm sgn}(i_k)}g_kt^{{\rm sgn}(i_{k+1})}$ can be closed in either  
$C_H$ (or $C_K$) by means of a path with the  
same end points which does not intersect  $B(cr)$.  
Here ${\rm sgn(i)}$ denotes the sign of the non-zero $i$.  
We obtain a loop lying in a copy of $\widetilde X_1$   
outside of $B(cr)$ which can  
therefore be contracted outside $B(r)$. Thus the loop $L$ is homotopic 
outside $B(r)$ to a new loop for which the quantity $\sum _{j=1}^n |i_j|$  
dropped-off by two units. The claim follows by induction.

\begin{remark} 
If $G_i$ are one-ended sci and $H$ is finitely generated  
multi-ended then $G_1\ast_H G_2$ is one-ended but {\em not} sci according to Jackson (see \cite{Ja2}).   
\end{remark}

\subsection{Proof of Proposition \ref{enddepth}} 
 The first step is the following lemma: 
 
\begin{lemma}\label{infgeod} 
In a homogeneous locally finite one-ended graph, through any point $p$ 
passes a discrete geodesic, i.e. an isometrically embedded copy of 
the integers. 
\end{lemma} 
 
\begin{proof} 
Since the graph is unbounded, for any $n\in \mathbb N$ there exist 
two vertices at distance $2n$, joined by a geodesic segment 
$u_n,u_{n-1},...,u_{-n}$. By homogeneity, we can choose as $u_0$ a 
 fixed  base point $u_0=x_0$. Now, this is true for any natural $n \in \mathbb 
N$, and since 
the graph is locally finite, there exists, by a compactness argument 
(e.g. diagonal extraction), the desired geodesic. 
\end{proof} 
 
Now, Proposition \ref{enddepth} follows from the following proposition. 
 
\begin{proposition} 
Let $X$ be a graph as before. Let $r\in \mathbb N$ be a natural 
number and $K$ be a finite subset of $X$ whose diameter is at most 
$2r$. Denote by $C$ a connected component of $X\sm K$. Then for any 
 point $x$ in $C$, we have the following alternative: 
 \begin{itemize} 
 \item either $x$ belongs 
to a geodesic ray (i.e. an embedded copy of the natural numbers) 
of $X$ within $C$ (and this in particular implies that $C$ is 
infinite), 
\item or else the distance from $x$ to $K$ is at most $r$ and  
 $C$ is bounded. 
\end{itemize} 
\end{proposition} 
\begin{proof} 
Let $x$ be a point of $C$. Then, by Lemma  \ref{infgeod}, there exists 
a discrete geodesic $(u_n)$ with $n\in \mathbb Z$ such that 
$u_0=x$. If $x$ does not belong to any geodesic ray contained in 
$C$, then one can find  $n$ and $m>0$ (both minimal)  such that 
$u_n$ and $u_{-m}$ belong to $K$. Since the diameter of $K$ is by 
hypothesis $\leq 2r$, then one has $m+n \leq 2r$. This means that 
the distance $d(x,K)$ from $x$ to $K$ is $\min \{ m, n \}\leq \frac 
{m+n}{2} $. Hence $x$ is within $r$ from $K$, and this ends the 
proof of the proposition. 
\end{proof}

\begin{proof}[End of proof of Proposition \ref{enddepth}] 
Whenever $K$ is a ball $B(r)$ of radius $r$ centered at the 
neutral element of the Cayley graph of the group $G$, then 
the previous proposition implies that any bounded connected 
component of $X_G\sm B(r)$ is included in the ball $B(2r)$ having the 
same center and radius $2r$. In particular one has that $V_0 (r) 
\leq 2r$. 
\end{proof} 
 
\begin{proof}[An alternative proof of Proposition \ref{enddepth}] 
Suppose that there is a positive integer $r\geqslant 2$, such that  
$V_0(r)>2r$. Then, there is a bounded connected component $A$   
of $X_G\smallsetminus B(r)$ and $a\in A$ such that $d(a,B(r))>r$.  
As $G$ is one-ended, there is an unbounded connected  
component  $C$ of $X_G\smallsetminus B(r)$. Consider the action of  
$a$ on $X_G$ by multiplication. Since $d(a,B(r))>r$, clearly  
$aB(r)=B(a,r)\subset A$ and there are $x\in A$, $y\in C$ so that  
$ax,ay\in B(r)$. Here $B(a,r)$ denotes the metric ball of radius $r$ centered at 
$a$.  
Therefore, there is a path $\gamma$ in $B(r)$ that  
joins $ax$ to $ay$. Then $a^{-1}\gamma$ is a path that joins an element  
of $A$ to an element of $C$, so it must pass through $B(r)$. Thus,  
there is $w$ on $\gamma$ so that $a^{-1}w\in B(r)$. This however  
implies that $w\in B(r)\cap a B(r)$ which is a contradiction. 
This proves that $V_0(r)\leqslant 2r$ and hence the end-depth of $G$ is linear. 
\end{proof}

\section{Proof of Theorem \ref{hrank}} 
 
Let $G$ be a connected, semisimple Lie group with trivial center 
and without compact factors. Unlike uniform lattices, non-uniform 
lattices $\Gamma$ in $G$ are not quasi-isometric to the symmetric 
space $X= G/K$ since they do not act cocompactly on $X$.  
But one can consider the following construction:  
chop off every cusp of the quotient 
$X/\Gamma$ and look at the lifts of each cusp to $X$, giving a 
$\Gamma$-equivariant union of horoballs in $X$. These horoballs 
are not disjoint in general; they can be made disjoint by cutting 
the cusps far enough out precisely when $\Gamma$ has $\mathbb 
Q$-rank one. The 
resulting space is called the {\em neutered space} $X_0$ 
associated to $\Gamma$, and $\Gamma$ acts cocompactly on it. 
The natural metric on $X_0$ is the path metric induced  
from $X$, given by the infimal length in $X$ of  
paths contained in $X_0$ that join the two points. 
Then $\Gamma$ endowed with the word metric is quasi-isometric  
to $X_0$ endowed with the path metric.   
However, sometimes the path metric on $X_0$ might be distorted  
with respect to the original metric on $X$. In order to circumvent this  
difficulty we consider first only higher rank groups.

\begin{proof}[Proof of Theorem \ref{hrank}] 
Since $G$ has higher rank, a result due to Lubotzky, Mozes, and Raghunathan  
(see \cite{LMR}) states that  
the embedding of $\Gamma$ endowed with the word metric into $G$ endowed with  
a left invariant metric is Lipschitz and hence a quasi-isometric embedding.  
The projection $G\to G/K$ is a quasi-isometry and hence  
$\Gamma$ is quasi-isometric to an orbit $\Gamma\cdot x_0\subset X$  
endowed with the restriction of the Riemannian metric $d_X$ on $X$.  
Finally the embedding of an orbit of $\Gamma$ into the neutered space $X_0$  
is a quasi-isometry when we consider the metric $d_X|_{X_0}$ on $X_0$.

By the quasi-isometry invariance of the sci growth,  
it will be sufficient to prove that $X_0$ endowed with the metric $d_X|_{X_0}$ 
has a linear $V_{X_0}$. The metric balls $B_{(X_0,d_X|_{X_0})}(x_0,r)$   
of radius $r$ centered at $x_0\in X_0$  for this non-geodesic metric  
are easy to describe, namely:  
\[ B_{(X_0,d_X|_{X_0})}(x_0,r)=B_{(X,d_X)}(x_0,r) \cap X_0\]  
in terms of the Riemannian metric balls $B_{(X,d_X)}(x_0,r)$.  
 
Now, the neutered space $X_0$ is obtained from $X$ by removing  
a collection of disjoint horoballs, as the $\mathbb Q$-rank of $\Gamma$  
is at least $2$.  Then any ball  $B_{(X,d_X)}(x_0,r)$ of $X$   
intersects only finitely many such horoballs. 
 
This implies that  the metric sphere $S_{(X_0,d_X|_{X_0})}(x_0,r)\subset  
\partial B_{(X_0,d_X|_{X_0})}(x_0,r)$  
is obtained from the usual metric sphere $\partial B_{(X,d_X)}(x_0,r)$ in $X$  
by removing from it the intersection with a disjoint  
union of finitely many horoballs.   
 
We need now a lemma which explains the geometry of such intersections:

\begin{lemma}\label{cat} 
Let $X$ be a proper CAT(0) manifold,   
$H$ be a horoball, and $B$ be a 
sphere of $X$. If the center $c$ of $B$ does not  
belong  to $H$, then $B\cap H$ is convex (i.e. topologically a ball). 
\end{lemma} 
\begin{proof} 
Let $f_c(x)=d(x,c)$ be the distance function to a fixed point $c\not\in H$.   
Then $f_c$ restricted to $H$ has only a critical point 
in $H$, namely the projection $p(c)$ of $c$ on $H$,  
where it achieves a non-degenerate minimum. 
Since $f_c$ is proper, the level sets on $H$   
retract onto $p(c)$. 
\end{proof} 
 
From Lemma \ref{cat} we derive that the metric spheres in $(X_0, d_X|_{X_0})$ are 
obtained from $S^{n-1}$ by removing finitely many  
disjoint disks $D^{n-1}$. 
This means that, whenever the dimension $n$ of $X$  
is $n\geq 4$, the metric spheres in $X_0$ are simply connected.  
It follows that $V_{(X_0, d_X|_{X_0})}(r)=r$  is linear and hence $\Gamma$  
has linear sci.

For the second part of the Theorem \ref{hrank} consider  
$\Gamma$ a non-uniform lattice in $SO(n,1)$.  
A non-uniform lattice $\Gamma$  
acts properly and cocompactly by 
isometries on $X_0=\mathbb H^n \sm  \mathcal F$ where $\mathcal F$ is a  
finite union of disjoint open horoballs.  
The result does not follow from Lemma \ref{cat}, as the metric  
on this truncated hyperbolic metric space is the path metric, which is  
exponentially distorted.    
Nevertheless this space is CAT(0) (by \cite{BrHa},  
Cor. 11.28 p.362 and \cite{Ru}).   
Metric balls are therefore homeomorphic to balls and their boundaries are  
spheres.   
In order to understand the topology of the metric spheres it suffices  
to consider a neighborhood of one horoball $H$. Given $c\in X_0$  
consider the cone in $\mathbb H^n$  
with vertex $c$ which is tangent to the horoball $H$ along an equidistant   
$(n-1)$-sphere $S^{n-1}(c)\subset \partial H$. If $p$  
belongs to the  visible $n$-disk bounded by $S^{n-1}$ on $\partial H$  
the geodesics segments joining $p$ and $c$ for the hyperbolic  
metric $d_{\mathbb H^n}$ and the path metric on $X_0$ coincide.    
When $p\in\partial H$ is outside the visible disk a geodesic  
segment in the path metric consists of a spherical segment $pq$ joining  
$p$ to $q\in S^{n-1}(c)$ followed by a geodesic segment $qc$.  
It follows that metric spheres in the path metric are obtained from  
a sphere by deleting a number of disjoint disks corresponding to  
visible disks at distance smaller than the radius.  
For $n\geq 4$ these are simply connected and this shows that  
$X_0$ with its path metric has linear sci.  
\end{proof}

\section{Other classes of groups with linear sci} 
Recall that a {\em Coxeter group} is a group $W$ with presentation of the 
following form:  
 
$$ \langle s_1, s_2 , \ldots , s_n | s_i^2 =1 
\mbox{ for } i \in \{1,2, \ldots ,n\}, (s_is_j)^{m_{ij}}=1\rangle $$  
where 
$i<j$ ranges over some subset of $\{1, 2, \ldots , n\}\times \{1, 
2, \ldots , n \}$ and $m_{ij}\geq 2$. Let $W$ be a Coxeter 
group. 
 
\begin{proposition}\label{Coxeter} 
Coxeter groups which are sci have linear sci.  
\end{proposition} 

\begin{proof} 
The Davis complex (see \cite{Da1}) $D_W$ of a   
finitely-generated Coxeter group $W$ is a CAT(0)  
cell complex $D_W$ on which  
$W$ acts on cellularly, properly, and with finite quotient.   
The links of vertices  of $D_W$ are all isomorphic to a fixed finite 
simplicial complex $L$, where $L$ can be described combinatorially  
in term of subsets of the generating set of $W$.  
 
It has been proved in \cite{DM} that a Coxeter group is sci if and only 
if its {\em nerve} $L$ and all its {\em punctured links} 
$L-\sigma$ are simply connected (where $\sigma$ is any simplex of $L$). 
The boundary of a metric ball in $D_W$ is a connected sum of 
various punctured links $L-\sigma$, and hence it is simply 
connected.  
 
Now any loop outside the metric ball  
of radius $r$ can be contracted onto the boundary of the  
metric ball and there contracted to a point.  
This implies that $V_{D_W}(r)=r$.   
 
The action of the Coxeter group on the 
Davis complex is not free but has finite stabilizers. Moreover 
there exists a finite index subgroup which acts freely on the 
Davis complex. This finite index subgroup is still sci and  
quasi-isometric to $D_W$ and hence by the previous 
arguments it has linear sci. This implies that $W$ has linear sci. 
\end{proof}

\begin{remark} 
The same proof shows that a sci right-angled Artin group has linear sci. 
More generally, Artin groups are semistable and have linear semistability 
(see \cite{Mi3}). 
\end{remark} 
 
\begin{remark} 
The connectivity of the punctured links determines the 
connectivity at infinity of $W$. However in  
\cite{DM2} the authors constructed a CAT(0) cell complex  
acted properly and cocompactly by $W$ whose nerve and  
punctured links are {\em not} simply connected, though as $W$ is sci.  
Thus the linear sci is the geometric property of groups  
which is closest to the ``simple connectivity of large spheres".   
\end{remark} 
 
\begin{remark} 
If $1 \to H \to G \to K \to 1$ is an exact sequence of finitely presented infinite groups 
where either $H$ or $K$ has one end then $G$ has a linear sci growth, from \cite{Ja2}. 
\end{remark}

\begin{remark} 
Mihalik and Tschantz have proved (see \cite{MiTsch,MiTsch3})  
that amalgamated products and HNN extensions of semistable groups over  
arbitrary finitely generated subgroups are semistable.    
We don't know whether Theorem \ref{amalgam} can be extended  
to multi-ended subgroups and linear semistability.    
\end{remark}


\begin{thebibliography}{HD}

%% Use the widest label as parameter.
%% Reference items can be numbered or have labels of your choice, as below.

%% In IMPAN journals, only the title is italicized; boldface is not used.
%% Our software will add links to many articles; for this, enclosing volume numbers in { } is helpful
%% Do not give the issue number unless the issues are paginated separately.

%%%%%%%%%%% To ease editing, use normal size:

\normalsize
\baselineskip=17pt

%%%%%%%%%%%%%

 \bibitem{bm} M. Bestvina and G. Mess, {\em The boundary of negatively curved groups}, J. Amer. Math. Soc. 4 (1991),  
469--481. 
 
 
\bibitem{bk} M. Bonk and B. Kleiner, {\em Quasi-hyperbolic planes in hyperbolic groups}, Proc. Amer. Math. Soc. 133 (2005), 2491--2494. 
 
\bibitem{bb} B. Bowditch,  
  {\em Connectedness properties of limit sets}, Trans. Amer. Math. Soc. 
 35 (1999), 3673--3686. 
 
\bibitem{Br1} 
S. G. Brick, {\em Quasi-isometries and ends of 
groups}, J. Pure Appl. Alg.  {86} (1993), 23--33. 
 
 
\bibitem{BrHa}  M. R. Bridson and A. Haefliger, Metric spaces of non-positive curvature,  Grundlehren Math. Wiss., vol. 319, Springer-Verlag, Berlin, 1999. 
 
 
 \bibitem{CR} 
S. Cleary and T. Riley, {\em A finitely presented group with 
unbounded dead-end depth},  
Proc. Amer. Math. Soc.  134  (2006),  343�-349,  
{\em Erratum}, Proc. Amer. Math. Soc. 136 (2008), 2641�-2645. 
 
 
 
\bibitem{Da1} 
M. W. Davis, {\em Groups generated by reflections and aspherical 
manifolds  not covered by Euclidean space}, Ann. of Math.  117 (1983), 293--324. 
 
\bibitem{DM} 
M. W. Davis and J. Meier, {\em The topology at infinity of Coxeter groups 
and buildings}, Comment. Math. Helv.  77 (2002), 746--766. 
 
 
\bibitem{DM2} 
M. W. Davis and J. Meier, {\em Reflection groups and CAT(0) complexes 
with exotic local structures}, World Sci. Publishing, River Edge, 
NJ, 2003, 151--158. 
 
 
\bibitem{FO} 
L. Funar and D. E. Otera, {\em  A refinement of the simple 
connectivity at infinity of groups}, Archiv Math. (Basel) 
81(2003), 360-368. 
 
\bibitem{G}
 R. Geoghegan, Topological methods in group theory, Graduate Texts in Math. 243, Springer, New York, 2008. 
 
 
\bibitem{GM} 
R. Geoghegan and M. L. Mihalik, {\em  
The fundamental group at infinity},   
Topology  35  (1996),   655--669.  
 
 
\bibitem{H} 
  C. H. Houghton, {\em Cohomology and the behaviour at infinity of finitely presented groups}, J. London Math. Soc. (2) 15 (1977),  465--471.
  
\bibitem{Ja2} 
B. Jackson, {\em End invariants of amalgamated free products}, J. 
Pure Appl. Alg. 23 (1982), 243--250. 
 
 
\bibitem{kb}I. Kapovich and N. Benakli, {\em Boundaries of hyperbolic groups}, in Combinatorial and geometric group theory, 39--93. 
Contemp. Math.  296, Amer. Math. Soc., Providence, RI, 2002.  
\bibitem{kl} 
B. Kleiner, {\em The asymptotic geometry of negatively curved spaces: uniformization, geometrization and rigidity}, Proc. I.C.M. Madrid, 2006, 743--768.  
 
 
 
 
\bibitem{LMR} 
A. Lubotzky, S. Mozes and M. S. Raghunathan, {\em The word and Riemannian metrics on lattices of semisimple groups},  Inst. Hautes Etudes Sci. Publ. Math.   
No. 91  (2000), 5--53. 
 
\bibitem{LS} 
R. C. Lyndon and P. E. Shupp, {Combinatorial group theory}, 
Springer-Verlag, Berlin, Heidelberg, New York, 1970. 
 
 
 
 
 \bibitem{Mi1}
 M. L. Mihalik, {\em  Ends of groups with the integers as quotient},
  J. Pure Appl. Alg. 35 (1985),  305--320.
 
 
\bibitem{MiTsch} 
M. L. Mihalik and S. T. Tschantz, {\em Semistability of amalgamated products, HNN-extensions, and all one-relator groups},   
Bull. Amer. Math. Soc. (N.S.)  26  (1992), 131--135. 
 

 
\bibitem{MiTsch3} 
M. L. Mihalik and S. T. Tschantz, {\em  
Semistability of amalgamated products and HNN-extensions},  
Mem. Amer. Math. Soc. 98 (1992), no. 471, vi+86 pp.  
 
 
\bibitem{Mi3} 
M. L. Mihalik, {\em Semistability of Artin and Coxeter groups}, J. Pure  
Appl. Alg.  11 (1996), 205--211. 
 
 
\bibitem{Ot-08} 
D. E. Otera, {\em Some remarks on the ends of groups}. 
 Acta Universitatis Apulensis  15 (2008), 133--146. 
 
 
 
\bibitem{Ru} 
K. Ruane, {\em CAT(0) boundaries of truncated hyperbolic space}. 
Spring Topology and Dynamical Systems Conference. 
Topology Proc. 29 (2005),  317--331. 
 
 \bibitem{Sie}
 L. C. Siebenmann, {\em The obstruction to finding a boundary for an open manifold of dimension greater than five}, 
 PhD thesis, Princeton University,  1965, 152 pp.
 
\bibitem{Sv} 
E. L. Swenson, {\em  
A cut point theorem for ${\rm CAT}(0)$ groups}, 
J. Diff. Geom. 53 (1999),  327--358.  
 
\bibitem{Sw} 
G. Swarup, {\em   
On the cut point conjecture},   
Electron. Res. Announc. Amer. Math. Soc.  2  (1996), 98--100. 
 
 
 
\end{thebibliography}
\end{document}